\documentclass[11pt,oneside]{amsart}
\usepackage{geometry}
\usepackage{amssymb,amsmath,amsthm}
\usepackage{geometry}
\theoremstyle{plain}
\usepackage{hyperref}
\usepackage{comment}
\usepackage{booktabs}
\usepackage{comment}
\usepackage{pgfplots}
\usepackage{pgfplots}
\pgfplotsset{width=10cm,compat=1.9}
\usepgfplotslibrary{external}

\theoremstyle{plain}

\newtheorem{lem}{Lemma}[section]

\newtheorem{lemma}[lem]{Lemma}

\newtheorem{proposition}[lem]{Proposition}

\newtheorem{thm-n}[lem]{Theorem}

\theoremstyle{remark}

\newtheorem*{rem*}{Remark}
\newtheorem*{notat*}{Notation}
\newtheorem*{exm*}{Example}
\theoremstyle{definition}

\usepackage{color}
\newcommand{\eilidh}[1]{\textcolor{red}{#1}}

\begin{document}

\title[Z{\'e}mor hash preimages]{Preimages for Z{\'e}mor's Cayley hash function}

\author{Eilidh McKemmie}
\address{Department of Mathematical Sciences, Kean University, Union, NJ, USA.}
\email{eilidh.mckemmie@kean.edu}
\author{Amol Srivastava}
\address{Department of Mathematics, Rutgers University, Piscataway, NJ, USA.}
\email{as3655@scarletmail.rutgers.edu}%

\begin{abstract}


In 1991, Z{\'e}mor proposed a hash function which provides data security using the difficulty of writing a given matrix as a product of generator matrices. Tillich and Z{\'e}mor subsequently provided an algorithm finding short collisions for this hash function.

We extend this collision attack to a stronger preimage attack, under the assumption that we can factor large integers efficiently. The Euclidean algorithm will factor a $2\times 2$ matrix with non-negative integer entries and determinant $1$. This factorization is short if the matrix entries are all roughly the same size. Therefore, to factor a matrix we need only find an integer matrix with the listed properties which is congruent to the target matrix modulo $p$; finding such an integer matrix is equivalent to solving a Diophantine equation. We give an algorithm to solve this equation.

\end{abstract}

\maketitle

\section{Introduction}
A hash function is a function that takes in a message of any size and produces a fixed-size string called a hash. Such functions with nice security properties are called \textbf{cryptographic hash functions}. These functions are used for a wide variety of cryptographic applications, for example password verification, digital signatures, and proof-of-work systems.

Desirable security properties include \textbf{collision resistance}, which means it is computationally infeasible to find two inputs with the same hashed output, and \textbf{preimage resistance}, which means it is computationally infeasible to find a message with a given hash.

A \textbf{Cayley hash function} is determined by a finite group $G$ and a set of group elements $g_1, ..., g_n \in G$. The input is a string with letters in $\{1, ..., n\}$, and the function takes the product of the corresponding group elements in order, mapping an arbitrary string into the finite set $G$ as follows \begin{align*}h:\{1, ..., n\}^* &\to G\\
h: (a_1, a_2, ..., a_m)&\mapsto g_{a_1}g_{a_2}\cdots g_{a_m}.\end{align*}

The name comes from the \textbf{Cayley graph} whose vertices are elements of the group $G$, with edges $s\to g_is$ corresponding to the set of generators. Security properties can be interpreted as graph-theoretic properties. For example, collision resistance of a Cayley hash function means it is computationally infeasible to find a cycle in the Cayley graph, while preimage resistance means it is computationally infeasible to find a path in the Cayley graph from the identity to a given group element.

One of the early Cayley hash function proposals came from Z{\'e}mor \cite{Zemor} using the group $SL_2(p)$ of $2\times 2$ matrices of determinant $1$ with entries modulo a prime $p$, inspired by Margulis' suggestion that the Cayley graph of $SL_2(p)$ might have good cryptographic properties \cite{Margulis_1982}.

Cayley hash functions have generated quite a bit of interest, leading to many proposals and attacks, for example \cite{Bromberg_2015,coz2022,Faugere2011,Grassl_2010,Petit_2008,Petit_2009,Petit_2011,Petit_2012,Petit_2013,Petit2016,Tillich,Tillich_2008,Tillich1994,tinani2023methods,Zemor,Zemor_1994}. As far as we are aware, the only unbroken proposals are \cite{Bromberg_2015,coz2022}.

Tillich and Z{\'e}mor \cite{Tillich1994} showed that Z{\'e}mor's Cayley hash function \cite{Zemor} is not collision resistant. We will show that, if we assume we can efficiently factor large integers, then it is not preimage resistant either. In particular, this says that Z{\'e}mor's Cayley hash function is not post-quantum preimage resistant.

We used GAP \cite{GAP4} to implement our algorithm, and our code is included in an ancillary file in the arXiv version of this paper.



\section{Preliminaries}

Z{\'e}mor's Cayley hash function \cite{Zemor} uses two generators, \[A=\begin{pmatrix}
    1&1\\0&1
\end{pmatrix}, \qquad B=\begin{pmatrix}
    1&0\\1&1
\end{pmatrix},\] chosen because they make computations efficient and generate all of $SL_2(p)$ for any prime $p$.  A \textbf{word} in $A$ and $B$ is a string of $A$s and $B$s, and the \textbf{product} of the word is simply the product of the matrices in order. Breaking collision and preimage resistance requires providing efficient algorithms to solve the following word problems.
\begin{itemize}
    \item Collision resistance: given a prime $p$, find a short word in $A$ and $B$ whose product is the identity matrix.
    \item Preimage resistance: given a prime $p$ and a matrix $M\in SL_2(p)$, find a short word in $A$ and $B$ whose product is $M$.
\end{itemize}

Tillich and Z{\'e}mor \cite{Tillich1994} gave an algorithm which efficiently factors the identity as a product of length $O(\log p)$. Their algorithm first solves a Diophantine equation in four variables $k_1, k_2, k_3, k_4$ to find a matrix with non-negative integer entries and determinant $1$ of the form \[\begin{pmatrix}
    1+k_1p&k_2p\\k_3p&1+k_4p
\end{pmatrix}\in SL(\mathbb{Z}).\]

Then they use the following proposition to factor the identity in $SL_2(p)$.

\begin{proposition}[{\cite[Theorem~4]{Tillich1994}}]\label{prop: Euclid}
    Let $M\in SL_2(\mathbb{Z})$ be chosen uniformly at random from the set of matrices with coefficients in the box $\{1, ..., N\}$.
    
    The Euclidean algorithm can be used to factor $M$ as a product of $A$ and $B$ in $O(\log N)$ time. This factorization has length $O(\log N)$ almost surely as $N\to \infty$.
\end{proposition}
For example, if $a>b>d\ge 0$ and $a>c\ge 0$ with $ad-bc=1$ then applying the Euclidean algorithm to $(a, b)$ yields quotients $q_1, q_2, ..., q_n$ such that \[A^{q_n}B^{q_{n-1}}\cdots A^{q_2}B^{q_1}=\begin{pmatrix}
    a&b\\c&d
\end{pmatrix},\] and other cases are similar.
 
The logarithmic length of the words  relies on the fact that the entries in their integer matrix are of roughly the same magnitude. A naive extension of their method to find a matrix of the form \[\begin{pmatrix}
    a+k_1p&b+k_2p\\c+k_3p&d+k_4p
\end{pmatrix}\] gives $k_1, k_2, k_3$ and $k_4$ of vastly different magnitudes when any one of $a, b, c$ or $d$ is large. This leads to exponentially long words, making this method infeasible. In the following section we describe a new method for factoring these matrices with word length $O((\log p)^2)$.

\section{Preimages}
Tillich and Z{\'e}mor provide a probabilistic algorithm for writing the identity matrix as a word of length $O(\log p)$ in $A$ and $B$, thereby showing that Z{\'e}mor's hash function is not collision resistant. This algorithm may also be used to write $A^{-1}$ and $B^{-1}$ as word of length $O(\log(p))$ in $A$ and $B$: if we have the identity as a word beginning with $A$ or $B$, then we simply delete the first letter to factor $A^{-1}$ or $B^{-1}$. Once we have factored one of $A^{-1}$ or $B^{-1}$, we factor the other using $B^{-1}=A^{-T}$ or $A^{-1}=B^{-T}$.

We will show that, in a post-quantum world, Z{\'e}mor's hash function is not preimage resistant. In particular, we will give an algorithm that, under the assumption that we can factor large integers, efficiently writes any matrix $M\in SL_2(p)$ as a word of length $O((\log p)^2)$ in $A$ and $B$ for any prime $p$.

Our algorithm will write $M$ as a word of length $O(\log p)$ in $A, B, A^{-1}$ and $B^{-1}$, then use Tillich and Z{\'e}mor's factorization of $A^{-1}, B^{-1}$ to get the final result.

Given the matrix $M$ to factor, the first step is to find the $LU$ decomposition to decompose $M$ into easily factorable parts $M=PLDU$ where $P$ is a permutation matrix, $L$ is lower unitriangular, $D$ is diagonal, and $U$ is upper unitriangular. 

The decomposition is as follows:
\[\begin{pmatrix}
    a&b\\c&d
\end{pmatrix}=\begin{cases}
    \begin{pmatrix}
        1&0\\ca^{-1} & 1
    \end{pmatrix}\begin{pmatrix}
        a&0\\0&a^{-1}
    \end{pmatrix}\begin{pmatrix}
        1&a^{-1}b\\0&1
    \end{pmatrix} & \text{when }a\ne 0\\
    \begin{pmatrix}
        0&1\\-1&0
    \end{pmatrix}\begin{pmatrix}
        -c&0\\0&-c^{-1}
    \end{pmatrix}\begin{pmatrix}
        1&c^{-1}d\\0&1
    \end{pmatrix} & \text{when }a=0
\end{cases}.\]

\subsection{Reduction to diagonal matrices}
We will reduce the general factorization problem to factoring a diagonal matrix, for which we will give an algorithm in the next section.

\begin{proposition}\label{prop: reduce to diagonal}
    The matrix $\begin{pmatrix}
        0&1\\-1&0
    \end{pmatrix}$ and any unitriangular matrix in $SL_2(p)$ can be factored as a word of length at most $6$ in $A, B, A^{-1}, B^{-1}$ and diagonal matrices.
\end{proposition}
\begin{proof}
    The permutation matrix may be written as the following product \[\begin{pmatrix}
        0&1\\-1&0
    \end{pmatrix}=A^{-1}BA^{-1}\begin{pmatrix}
        -1&0\\0&-1
    \end{pmatrix}.\]

    If $b$ is a quadratic residue modulo $p$ then $b=a^2$ for some $a$, in which case \[\begin{pmatrix}
        1&0\\b&1
    \end{pmatrix}=\begin{pmatrix}
        a^{-1}&0\\0&a
    \end{pmatrix}B\begin{pmatrix}
        a&0\\0&a^{-1}
    \end{pmatrix}.\]

    If $b$ is a quadratic non-residue modulo $p$ then we can write $b=r^2-s^2$ where $r\equiv 2^{-1}(b+1) \bmod p$ and $s\equiv 2^{-1}(b-1) \bmod p$, in which case
    \[\begin{pmatrix}
        1&0\\b&1
    \end{pmatrix}=\begin{pmatrix}
        1&0\\r^2&1
    \end{pmatrix}\begin{pmatrix}
        1&0\\s^2&1
    \end{pmatrix}^{-1}=\begin{pmatrix}
        r^{-1}&0\\0&r
    \end{pmatrix}B\begin{pmatrix}
        r&0\\0&r^{-1}
    \end{pmatrix}\begin{pmatrix}
        s^{-1}&0\\0&s
    \end{pmatrix}B^{-1}\begin{pmatrix}
        s&0\\0&s^{-1}
    \end{pmatrix}.\]

    For upper unitriangular matrices we transpose lower unitriangular matrices \[\begin{pmatrix}
        1&b\\0&1
    \end{pmatrix}=\begin{pmatrix}
        1&0\\b&1
    \end{pmatrix}^T,\] and note that since $A^T=B$ the above formulae give $\begin{pmatrix}
        1&b\\0&1
    \end{pmatrix}$ as a word of length at most $6$.
\end{proof}

\subsection{Factoring diagonal matrices}

This section shows how to factor a diagonal matrix $\begin{pmatrix}
    a&0\\0&a^{-1}
\end{pmatrix}$ as a word of length $O(\log(p))$ in $A, B, A^{-1}$ and $B^{-1}$.

We will need the following lemma.
\begin{lemma}\label{lem: ad}
Given a prime $p$ and an integer $0<a<p$, one can efficiently find an integer $d\equiv a^{-1} \bmod p$ such that $a$ and $d$ are coprime integers of similar magnitude. In particular, we find $d\sim p (\log p )(\log \log p)$ in the worst case and $d\sim p (\log \log p) (\log \log \log p)$ on average.
\end{lemma}
\begin{proof}
    First use the Extended Euclidean Algorithm to find $0<d<p$ such that $d\equiv a^{-1} \bmod p$. Now we find an integer $k$ such that $d+kp$ and $a$ are coprime. Let $q_1, ..., q_m$ be the distinct prime factors of $a$. We need to choose $k$ such that, for all $i$, we have $d+kp\not \equiv 0 \bmod q_i$. This means avoiding one congruence class modulo $q_i$ for each $i$, which can be done with $k<p_{m+1}$ where $p_m$ is the $m$th smallest prime.

    Since $m$ is the number of distinct prime divisors of $a$, we know that $m\le \log a\le \log p$, and on average $m \sim \log \log p$. By the Prime Number Theorem, $p_{m+1}\sim m \log m$. Therefore in the worst case $k\sim (\log p) (\log \log p)$, so one can find $k$ efficiently by sequential search.

    Then $a\sim p$ and $d+kp\sim p (\log p) (\log \log p)$.
\end{proof}

\begin{proposition}
    Let $a$ be an integer such that $1<a<p$. Assume that we can factor integers close to $p$ efficiently. Then there is an efficient algorithm which factors $\begin{pmatrix}
        a&0\\0&a^{-1}
    \end{pmatrix}$ as a word in $A$, $B$, $A^{-1}$ and $B^{-1}$ of length $O(\log p)$.
\end{proposition}
\begin{proof}
    Let $d\equiv a^{-1} \bmod p$ such that $a$ and $d$ are coprime integers of similar magnitude, which we can find efficiently by Lemma~\ref{lem: ad}. 
    We will find $k_1, k_2, k_3, k_4$ such that $1=\det\begin{pmatrix}
        a+k_1p&k_2p\\k_3p&d+k_4p
    \end{pmatrix},$ and each matrix entry is of roughly the same magnitude.

    We must solve the Diophantine equation
    \[n+ak_4+dk_1=(k_1k_4-k_2k_3)p,\] so we will first solve $n+ak_4+dk_1 \equiv 0 \bmod p$, and then find suitable $k_2, k_3$.
    
    Since $d$ and $a$ are coprime, we may invert $d$ modulo $a$, and so find
    \begin{align*}
        k_1&\equiv -nd^{-1} \bmod a\\
        k_4&\equiv (-n-dk_1)d \bmod p.
    \end{align*}

    To ensure $a+k_1p, k_2p, k_3p, d+k_4p$ are all of similar magnitude, we will pick the $k_i$ to be of magnitude $\sim p^2$. We do this randomly because in some cases the Euclidean algorithm will fail, but succeeds after a few random tries. In our GAP code we use
    \begin{verbatim}
        k1:=-n/d mod a + a*p;
        k4:=(-n-d*k1)/a mod p + p*Random(source, 1, p);
    \end{verbatim}

    Once we have chosen $k_1, k_4$ we must choose $k_2, k_3$ of magnitude roughly $p^2$ such that \[k_2k_3=k_1k_4+\frac{n+ak_4+dk_1}{p},\]
    which we do by factoring, assuming we can do this efficiently.

     Then, since the matrix entries are all of magnitude $p^2$, Proposition~\ref{prop: Euclid} says we may use the Euclidean algorithm to write $\begin{pmatrix}
        a+k_1p&k_2p\\k_3p&d+k_4p
    \end{pmatrix}$ as a product of $A$ and $B$, thereby factoring our matrix $\begin{pmatrix}
        a&0\\0&a^{-1}
    \end{pmatrix}$ upon reduction modulo $p$.
\end{proof}

\section{Implementation and factorization length}

We used GAP \cite{GAP4} to implement our algorithm, and our code is included in an ancillary file in the arXiv version of this paper. The code is meant as a proof of concept, and the efficiency has not been fully optimized. We made the following choices in our implementation.

\begin{itemize}
    \item When implementing Tillich and Z{\'e}mor's algorithm, we try $100(\log p)^2$ times to find an appropriate value of $c$ before returning \verb!fail!, resulting in a small error rate. We tested our implementation for 1000 random $n$-bit primes and found that for $n>20$ the error rate vanishes.
    \item To factor one matrix, our algorithm factors several diagonal matrices. For each diagonal matrix we factor, we make at most 15 attempts to factor it as a word of length at most $1000 \log p$ in $A, B, A^{-1}, B^{-1}$, otherwise we take the shortest attempt.
\end{itemize}

\subsection{The length of the factorizations}

Proposition~\ref{prop: reduce to diagonal} writes $M$ as a word of length at most $6$ in $A, B, A^{-1}, B^{-1}$ and diagonal matrices. We use Tillich and Z{\'e}mor's algorithm to factor $A^{-1}$ and $B^{-1}$ as words of length $O(\log p)$ in $A, B$. When factoring a diagonal matrix $\begin{pmatrix}
    a&0\\0&d
\end{pmatrix}$, we use a function \verb!findIntMatrixD! which finds a matrix $\begin{pmatrix}
        a+k_1p&k_2p\\k_3p&d+k_4p
    \end{pmatrix}\in SL_2(\mathbb{Z})$ with entries $\sim p^2$. Assuming \verb!findIntMatrixD! produces uniformly random matrices with entries in $\{1, ..., p^2\}$, Proposition~\ref{prop: Euclid} finishes the factorization of a diagonal matrix as a word in $A$ and $B$ of length $O( \log p)$. In this case we should be able to factor any matrix as a word of length $O(\log p)$ in $A$ and $B$.

    In implementing \verb!findIntMatrixD!, we found that the outputs are not as close to uniformly random as we would like. This seems to be because we are looking for matrices whose entries are $\sim p^2$. If we were to instead search for a matrix whose entries are significantly larger, we would see more uniform outputs, at the expense of having to factor larger integers.
    
    With our implementation of \verb!findIntMatrixD! we tend to see some very large powers of $A$ and $B$ as factors, as well as the very small powers we would expect with true randomness. Therefore we use $A^{p-k}=(A^{-1})^k$ and $B^{p-k}=(B^{-1})^k$ where $0<k<\frac{p}{2}$, writing our diagonal matrices as words in $A, B, A^{-1}, B^{-1}$.

    We ran some experiments on a standard laptop computer running GAP, generating random primes $p$ of $10, 20, 40$ and $80$ bits, and random matrices in $SL_2(p)$ to factor. The prime $p$ is small only because we need to factor integers of size $\sim p^2$.
    
    Our experiments below show that our algorithm can factor a random matrix as a word in $A, B, A^{-1}, B^{-1}$ of length $O(\log p)$, which gives a word in $A, B$ of length $O((\log p)^2)$. We list here the word length divided by $(\log p)^2$.

\vspace{1cm}

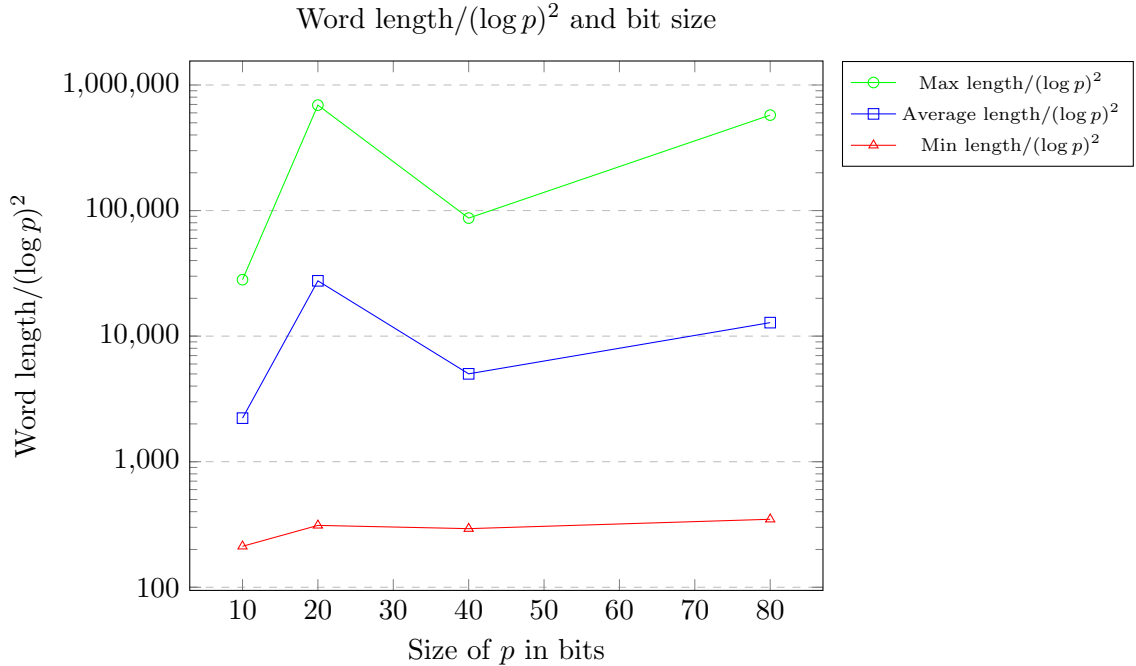
\begin{tabular}{clccc}\toprule
     Bits  &10& 20&  40& 80\\\midrule
  Number of tests&100& 100& 100& 100\\
  Average length/$(\log p)^2$&2222.50& 27571.24& 5009.12& 12807.78\\
  Min length/$(\log p)^2$&212& 311& 293& 348\\ 
 Max length/$(\log p)^2$& 28153& 691453& 87087& 575764\\
 Average runtime (in ms)& 2.18& 36.56& 138.91&10007.67\\ \bottomrule
\end{tabular}

\vspace{1cm}

\begin{tikzpicture}
\begin{axis}[
 title={{Word length/$(\log p)^2$ and bit size}},
    xlabel={Size of $p$ in bits},
    ylabel={{Word length/$(\log p)^2$}},
    legend pos=outer north east,
    ymajorgrids=true,
    grid style=dashed,
    ymode=log,
    log ticks with fixed point,
    legend style={scale uniformly calculation=0.5}
    x filter/.code=\pgfmathparse{#1 + 6.90775527898214},
]
\addplot[
    color=green,
    mark=o,
    ]
    coordinates {
    (10,28153)(20,691453)(40,87087)(80,575764)
    };
    \addlegendentry{\tiny{Max length/$(\log p)^2$}}
\addplot[
    color=blue,
    mark=square,
    ]
    coordinates {
    (10,2222.5)(20,27571.24)(40,5009.12)(80,12807.78)
    };
    \addlegendentry{\tiny{Average length/$(\log p)^2$}}
\addplot[
    color=red,
    mark=triangle,
    ]
    coordinates {
    (10,212)(20,311)(40,293)(80,348)
    };
    \addlegendentry{\tiny{Min length/$(\log p)^2$}}
\end{axis}
\end{tikzpicture}

\vspace{1cm}

Our data and plots provide evidence that the word length grows like $O((\log p)^2)$.

\bibliography{bib}
\bibliographystyle{plain}

\end{document}